\documentclass[12pt,reqno]{amsart}
\usepackage{enumerate, latexsym, amsmath, amsfonts, amssymb, amsthm, color}
\allowdisplaybreaks[4]
\def\pmod #1{\ ({\rm{mod}}\ #1)}
\def\Z{\Bbb Z}
\def\N{\Bbb N}

\def\bg{\bigg}
\def\({\bg(}
\def\){\bg)}

\def\sgn{{\rm sgn}}

\theoremstyle{plain}
\newtheorem{theorem}{Theorem}

\newtheorem{lemma}{Lemma}
\newtheorem{corollary}{Corollary}

\theoremstyle{definition}

\theoremstyle{remark}

\def\cG{{\mathcal G}}
\def\cK{{\mathcal K}}

 \vspace{4mm}

\begin{document}

\hbox{}
\medskip

\title
[{}]
{On the supercongruences involving harmonic numbers of order $2$}

\author{Guo-Shuai Mao}
\address {School of Mathematics and Statistics, Nanjing University of Information Science and Technology, Nanjing 210044, People's Republic of China}
\email{maogsmath@163.com}

\author{Hao Pan}
\address{School of Applied Mathematics, Nanjing University of Finance and Economics, Nanjing 210023, People's Republic of China}
\email{haopan79@zoho.com}

\keywords{Congruences; Central binomial coefficients; Harmonic numbers; Bernoulli numbers.
\newline \indent The first author was supported by the Natural Science Foundation (Grant No. 12001288) of China and the second author was funded by the Natural Science Foundation (Grant No. 12071208) of China.}
\begin{abstract}
We prove several supercongruences involving the harmonic number of order two $H_n^{(2)}:=\sum_{k=1}^n1/k^2$. For example, if $p>5$ is prime and $\alpha$ is $p$-integral, then we can completely determine 
$$
\sum_{k=0}^{p-1}\frac{H_k^{(2)}}{k}\cdot\binom{\alpha}{k}\binom{-1-\alpha}{k}\quad\text{and}\quad
\sum_{k=0}^{\frac{p-1}{2}}\frac{H_k^{(2)}}{k}\cdot\binom{\alpha}{k}\binom{-1-\alpha}{k}
$$
modulo $p^3$. In particular, by setting $\alpha=-1/2$, we confirm two conjectured congruences of Z.-W. Sun \cite{S11a}.
\end{abstract}
\maketitle
\section{Introduction}
\setcounter{lemma}{0}
\setcounter{theorem}{0}
\setcounter{corollary}{0}
\setcounter{remark}{0}
\setcounter{equation}{0}
\setcounter{conjecture}{0}
For a positive integer $r$, define the $n$-th harmonic number of order $r$
$$
H_n^{(r)}:=\sum_{k=1}^n\frac{1}{k^{r}}.
$$
In particular, set $H_0^{(r)}=0$.
When $r=1$, $H_n:=H_n^{(1)}$ is the $n$-th harmonic number.
The harmonic numbers of high order have some interesting arithmetical properties. For example, 
for any prime $p>r+2$,
we have \cite{H} 
\begin{align}\label{Hp1r}
H_{p-1}^{(r)}\equiv\begin{cases}-\frac{r(r+1)}{2(r+2)}p^2B_{p-r-2}\pmod{p^3},&\text{if $r$ is odd},\\
\frac{r}{r+1}p B_{p-r-1}\pmod{p^2},&\text{if $r$ is even},\end{cases}
\end{align}
where the Bernoulli number $B_n$ is given by
$$
\sum_{n=0}^\infty\frac{B_n}{n!}t^n=\frac{t}{e^t-1}.
$$
Similarly, it is known \cite{s2000} that
\begin{align}\label{Hp12r}
H_{(p-1)/2}^{(r)}\equiv\begin{cases}-2q_p(2)\pmod{p},&\text{if $r=1$},\\
-\frac{2^r-2}{r}B_{p-r} \pmod{p},&\text{if $r>1$ is odd},\\
\frac{r(2^{r+1}-1)}{2(r+1)}pB_{p-r-1} \pmod{p^2},&\text{if $r$ is even},\end{cases}
\end{align}
for any prime $p>r+2$, where $q_p(a)=(a^{p-1}-1)/p$ stands for the Fermat quotient.

In \cite{S11a}, motivated by the convergent series concerning $\pi^3$, Z.-W. Sun proposed many curious conjectural congruences involving $H_k$ and $H_{k}^{(2)}$. Two of those conjectures are \cite[Conjecture 5.3]{S11a}
\begin{equation}\label{1.3a}
\sum_{k=1}^{p-1}\frac{H_{k}^{(2)}}{k16^k}\cdot \binom{2k}k^2\equiv-12\frac{H_{p-1}}{p^2}+\frac{7}{10}p^2B_{p-5}\pmod{p^3},
\end{equation}
and
\begin{equation}\label{1.3}
\sum_{k=\frac{p+1}2}^{p-1}\frac{H_{k}^{(2)}}{k16^k}\cdot\binom{2k}k^2\equiv\frac{31}2p^2B_{p-5}\pmod{p^3},
\end{equation}
where $p>3$ is prime.
In this paper, we shall confirm Sun's conjectures \eqref{1.3a} and \eqref{1.3}.
\begin{theorem}\label{SunCT}
\eqref{1.3a} and \eqref{1.3} are true.
\end{theorem}
For a prime $p$, let $\Z_p$ denote the ring of all $p$-adic integers.
For any $x\in\Z_p$, let $\langle x\rangle_p$ denote the least non-negative residue of $x$ modulo $p$, i.e.,
$\langle x\rangle_p\in\{0,1,\ldots,p-1\}$ and $x\equiv\langle x\rangle_p\pmod{p}$. 
In \cite{SZH14}, Z.-H. Sun proved that
\begin{equation}\label{ZHSun}
\sum_{k=0}^{p-1}\binom{\alpha}{k}\binom{-1-\alpha}{k}\equiv(-1)^{\langle\alpha\rangle_p}\pmod{p^2}
\end{equation}
for any odd prime $p$ and $\alpha\in\Z_p$. In particular, since $\binom{2k}{k}=(-4)^k\binom{-\frac12}{k}$, substituting $\alpha=-1/2$ in \eqref{ZHSun}, we get
$$
\sum_{k=0}^{p-1}\frac1{16^k}\cdot\binom{2k}{k}^2\equiv(-1)^{\frac{p-1}{2}}\pmod{p^2},
$$
which was conjectured by Rodriguez-Villegas \cite{RV03} and confirmed by Mortenson \cite{Mo04}. 
For the related results, the reader may refer to \cite{SZH16,SZH20}.

In \cite{Su14}, Z.-W. Sun completely determined
$$
\sum_{k=0}^{p-1}\binom{\alpha}{k}\binom{-1-\alpha}{k}H_k\quad\text{and}\quad\sum_{k=0}^{p-1}\binom{\alpha}{k}\binom{-1-\alpha}{k}H_k^{(2)}
$$
modulo $p^2$. For example, Sun \cite[(1.14)]{Su14} showed that
\begin{equation}
\sum_{k=0}^{p-1}\binom{\alpha}{k}\binom{-1-\alpha}{k}H_k^{(2)}\equiv-E_{p^2-p-2}(-\alpha)\pmod{p^2},
\end{equation}
where the Euler polynomial $E_n(x)$ is given by
$$
\sum_{n=0}^{\infty}\frac{E_n(x)}{n!}t^n=\frac{2e^{xt}}{e^{2t}+1}.
$$

Motivated by the results, we shall determined
$$
\sum_{k=1}^{p-1}\frac{H_k^{(2)}}{k}\cdot\binom{\alpha}k\binom{-1-\alpha}k\quad\text{and}\quad
\sum_{k=1}^{\frac{p-1}2}\frac{H_k^{(2)}}{\alpha+k}\cdot\binom{\alpha}k\binom{-1-\alpha}k
$$
modulo $p^3$, and give the following extension of Theorem \ref{SunCT}.
\begin{theorem}\label{Th1.1} Suppose that $p>5$ is a prime and $\alpha\in\mathbb{Z}_p$. Let $a=\langle \alpha\rangle_p$ and $t:=(\alpha-a)/p$. Then
\begin{align}\label{1.1}
\sum_{k=1}^{p-1}\frac{H_k^{(2)}}{k}\cdot\binom{\alpha}k\binom{-1-\alpha}k\equiv 2p^2t^2B_{p-5}-\frac{2}5p^2tB_{p-5}+\cG(a,t)\pmod{p^3},
\end{align}
where
$$
\cG(a,t):=-2H_a^{(3)}+6ptH_a^{(4)}+2p^2t(1-5t)H_{a}^{(5)}.
$$
Moreover, if $a\leq (p-1)/2$, then
\begin{align}\label{1.2}
\sum_{k=1}^{\frac{p-1}2}\frac{H_k^{(2)}}{k}\cdot\binom{\alpha}k\binom{-1-\alpha}k
\equiv4p^2t^2B_{p-5}-\frac{31}5p^2tB_{p-5}+\cK(a,t)\pmod{p^3},
\end{align}
where
$$
\cK(a,t):=-2H_a^{(3)}+8ptH_a^{(4)}-20p^2t^2H_a^{(5)}+2p^2t\sum_{k=1}^{a}\frac{2H_{2k}-H_k}{k^4}.
$$
\end{theorem}

Similarly, in the left sides of  \eqref{1.1} and \eqref{1.2} we can replace the denominators $k$ by $\alpha+k$  and obtain some congruences modulo $p^{4}$.
\begin{theorem}\label{Th1.3} Suppose that $p>5$ is a prime, $\alpha\in\mathbb{Z}_p$ and $p\nmid \alpha$. Let $a=\langle \alpha\rangle_p$ and $t:=(\alpha-a)/p$. Then
\begin{align}\label{1.6}
&\sum_{k=1}^{p-1}\frac{H_k^{(2)}}{\alpha+k}\cdot\binom{\alpha}k\binom{-1-\alpha}k\notag\\
\equiv&-\frac1{\alpha^3}+\frac{p^2t(t+1)}{\alpha\cdot a^2}\left(\frac1{\alpha^2}+\frac{2pH_a}{a^2}-\frac{p(2t+1)}{\alpha^2\cdot a}+\frac23pB_{p-3}\right)\pmod{p^4}.
\end{align}
And if $a\leq (p-1)/2$, then
\begin{align}\label{1.7}
&\sum_{k=1}^{\frac{p-1}2}\frac{H_k^{(2)}}{\alpha+k}\cdot\binom{\alpha}k\binom{-1-\alpha}k\notag\\
&\equiv-\frac1{\alpha^3}+\frac{pt}{\alpha\cdot a}\left(\frac1{\alpha^2}+\frac73pB_{p-3}-\frac{pt}{\alpha^2\cdot a}+\frac{2p}{a^2}\sum_{k=1}^{a}\frac1{2k-1}\right)\pmod{p^3}.
\end{align}
\end{theorem}
In particular, substituting $\alpha=-1/2$ in Theorem \ref{Th1.3}, we get
\begin{corollary}
$$
\sum_{k=1}^{p-1}\frac{H_k^{(2)}}{(2k-1)16^k}\cdot\binom{2k}k^2\equiv4+p^2\left(4+8p-16pq_p(2)+\frac23pB_{p-3}\right)\pmod{p^4},
$$
$$
\sum_{k=1}^{\frac{p-1}2}\frac{H_k^{(2)}}{(2k-1)16^k}\cdot`\binom{2k}k^2\equiv4-p\left(4+8pq_p(2)+\frac73pB_{p-3}\right)\pmod{p^3}.
$$
\end{corollary}
At the end of the section, let us introduce the notion of alternating multiple harmonic sum,
which will be used in our proofs of Theorems \eqref{Th1.1}-\eqref{Th1.3}. For non-zero integers $r_1,\ldots,r_m$, define
$$
H_n^{(r_1,\ldots,r_m)}:=\sum_{\substack{1\leq k_1<k_2<\ldots<k_m\leq n}}\prod_{i=1}^m\frac{\sgn(r_i)^{k_i}}{k_i^{|r_i|}},
$$
where $\sgn(r)$ denotes the sign of $r$. For example, 
$$
H_n^{(-1)}=\sum_{k=1}^n\frac{(-1)^k}{k},\qquad
H_n^{(1,-2)}=\sum_{k=2}^n\frac{(-1)^k}{k^2}\sum_{j=1}^{k-1}\frac{1}{j}.
$$
For the properties and applications of alternating multiple harmonic sums, the reader may refer to \cite{H, HHT, Tjnt, TZ}.

We are going to give the proof of Theorem \ref{Th1.1} in Section 2. Then in Section 3, with the  help of 
Theorem \ref{Th1.1} and several additional lemmas, we shall confirm Sun's conjectures (\ref{1.3a}) and (\ref{1.3}).
Finally, Section 4 is devoted to proving Theorem \ref{Th1.3}.

\section{Proof of Theorem \ref{Th1.1}}
\setcounter{lemma}{0}
\setcounter{theorem}{0}
\setcounter{corollary}{0}
\setcounter{remark}{0}
\setcounter{equation}{0}
\setcounter{conjecture}{0}
\begin{lemma}\label{Lem2.2} For any positive integer $n$, we have
\begin{equation}\label{id}
\sum_{k=1}^n\binom{x}{k}\binom{-x}kH_k^{(2)}=-\frac1{x^2}+\binom{x-1}{n}\binom{-x-1}n\left(\frac{1}{x^2}+H_n^{(2)}\right).
\end{equation}
\end{lemma}
\begin{proof} Let $f(n)$ and $g(n)$ denote the left-hand side and the right-hand side of the identity. It is easy to check that
\begin{align*}
f(n)-f(n-1)&=\sum_{k=1}^n\binom{x}{k}\binom{-x}kH_k^{(2)}-\sum_{k=1}^n\binom{x}{k}\binom{-x}kH_k^{(2)}\\
=&\binom{x}n\binom{-x}nH_n^{(2)}.
\end{align*}
And by noting that
$$\binom{x-1}n\binom{-x-1}n=\binom{x}n\binom{-x}n\cdot\frac{x^2-n^2}{x^2}$$
and
$$\binom{x-1}{n-1}\binom{-x-1}{n-1}=-\frac{n^2}{x^2}\cdot\binom{x}n\binom{-x}n,$$
we have
\begin{align*}
&g(n)-g(n-1)\\
=&\binom{x}{n}\binom{-x}n\cdot\frac{x^2-n^2}{x^2}\left(\frac1{x^2}+H_n^{(2)}\right)+\binom{x}{n}\binom{-x}{n}\cdot\frac{n^2}{x^2}\left(\frac1{x^2}+H_{n-1}^{(2)}\right)\\
=&\binom{x}{n}\binom{-x}nH_n^{(2)}.
\end{align*}
Hence
$$
f(n)-f(n-1)=g(n)-g(n-1),
$$
and $f(1)=g(1)=-x^2$. so by induction we can get $f(n)=g(n)$ for all $n\geq 1.$
\end{proof}
\begin{lemma}\label{Lem2.3} Let $p>3$ be an odd prime and let $t\in\mathbb{Z}_p$. If $1\leq k\leq p-1$, then
\begin{align*}
&\binom{pt+k-1}{p-1}\binom{-pt-k-1}{p-1}\\
\equiv&\frac{p^2t(t+1)}{k^2}\left(1+2pH_k-\frac{p}k-\frac{2pt}k\right)\pmod{p^4}.
\end{align*}
If $1\leq k\leq (p-1)/2$, then
\begin{align*}
\binom{pt+k-1}{\frac{p-1}2}\binom{-pt-k-1}{\frac{p-1}2}\equiv\frac{pt}{k}\left(1-\frac{pt}k+2pH_{2k}-pH_k\right)\pmod{p^3}.
\end{align*}
\end{lemma}
\begin{proof} It is easy to check that
\begin{align*}
&\binom{pt+k-1}{p-1}=\frac{(pt+k-1)\cdots(pt+1)pt(pt-1)\cdots(pt+k-p+1)}{(p-1)!}\\
\equiv&\frac{pt(k-1)!(1+ptH_{k-1})(-1)^{p-1-k}(p-1-k)!(1-ptH_{p-1-k})}{(p-1)!}\\
\equiv&\frac{pt}k\left(1+pH_k-\frac{pt}k\right)\pmod{p^3}.
\end{align*}
And by (\ref{Hp1r}), we have
\begin{align*}
&\binom{-pt-k-1}{p-1}\\
=&\frac{(pt+k+1)\cdots(pt+p-1)p(t+1)(pt+p+1)\cdots(pt+p+k-1)}{(p-1)!}\\
\equiv&\frac{p(t+1)(p-1)!(1+pt(H_{p-1}-H_k))(k-1)!(1+p(t+1)H_{k-1})}{k!(p-1)!}\\
\equiv&\frac{p(t+1)}k\left(1+pH_{k-1}-\frac{pt}k\right)\pmod{p^3}.
\end{align*}
Hence
$$
\binom{pt+k-1}{p-1}\binom{-pt-k-1}{p-1}\equiv\frac{p^2t(t+1)}{k^2}\left(1+2pH_k-\frac{p}k-\frac{2pt}k\right)\pmod{p^4}.
$$
Similarly,
\begin{align*}
&\binom{pt+k-1}{\frac{p-1}2}=\frac{(pt+k-1)\cdots(pt+1)pt(pt-1)\cdots(pt+k-\frac{p-1}2)}{(\frac{p-1}2)!}\\
\equiv&\frac{pt(k-1)!(1+ptH_{k-1})(-1)^{\frac{p-1}2-k}(\frac{p-1}2-k)!(1-ptH_{\frac{p-1}2-k})}{(\frac{p-1}2)!}\\
\equiv&\frac{pt}{k\binom{\frac{p-1}2}k}(-1)^{\frac{p-1}2-k}\left(1+ptH_{k-1}-ptH_{\frac{p-1}2-k}\right)\pmod{p^3}
\end{align*}
and
\begin{align*}
&\binom{-pt-k-1}{\frac{p-1}2}=\frac{(-1)^{\frac{p-1}2}(pt+k+1)\cdots(pt+k+\frac{p-1}2)}{(\frac{p-1}2)!}\\
\equiv&(-1)^{\frac{p-1}2}\binom{\frac{p-1}2+k}{k}\left(1+ptH_{\frac{p-1}2+k}-ptH_k\right)\pmod{p^2}.
\end{align*}
So
\begin{align*}
&\binom{pt+k-1}{\frac{p-1}2}\binom{-pt-k-1}{\frac{p-1}2}\\
\equiv&\frac{pt(-1)^k\binom{\frac{p-1}2+k}{k}}{k\binom{\frac{p-1}2}k}\left(1-\frac{pt}k+ptH_{\frac{p-1}2+k}-ptH_{\frac{p-1}2-k}\right)\pmod{p^3}.
\end{align*}
It is easy to see that $$(-1)^k\binom{\frac{p-1}2+k}{k}\binom{\frac{p-1}2}k\equiv\frac{\binom{2k}k^2}{16^k}\pmod{p^2}$$ and
$$
\binom{\frac{p-1}2}k^2\equiv\frac{\binom{2k}k^2}{16^k}(1-p(2H_{2k}-H_k))\pmod{p^2}.
$$
Thus,
\begin{align*}
\frac{(-1)^k\binom{\frac{p-1}2+k}{k}}{k\binom{\frac{p-1}2}k}=\frac{(-1)^k\binom{\frac{p-1}2+k}{k}\binom{\frac{p-1}2}k}{k\binom{\frac{p-1}2}k^2}\equiv1+p(2H_{2k}-H_k)\pmod{p^2}.
\end{align*}
Therefore by the fact that $H_{p-1-k}\equiv H_k\pmod p$ for each $0\leq k\leq p-1$, we have
$$
\binom{pt+k-1}{\frac{p-1}2}\binom{-pt-k-1}{\frac{p-1}2}\equiv\frac{pt}k\left(1-\frac{pt}k+p(2H_{2k}-H_k)\right)\pmod{p^3}.
$$
The proof of Lemma \ref{Lem2.3} is complete.
\end{proof}
\begin{lemma}
Suppose that $r,s$ are positive integers and $r+s$ is odd. For any prime $p>r+s+1$,
\begin{equation}\label{Hp1rs}
H_{p-1}(r,s)\equiv\frac{(-1)^sB_{p-r-s}}{r+s}\cdot\binom{r+s}{r}\pmod{p},
\end{equation}
and
\begin{equation}\label{Hp12rs}
H_{\frac{p-1}{2}}(r,s)\equiv\frac{B_{p-r-s}}{2(r+s)}\cdot\left((-1)^s\binom{r+s}{r}+2^{r+s}-2\right)\pmod{p}.
\end{equation}
\end{lemma}
\begin{proof}
(\ref{Hp1rs}) was proved in \cite{H} and (\ref{Hp12rs}) follows from \cite[Lemma 1]{HHT}.
\end{proof}

Define
\begin{equation}
S_n(x):=\sum_{k=1}^n\frac{H_k^{(2)}}k\cdot\binom{x}k\binom{-1-x}k.
\end{equation}
 \begin{lemma}\label{Lem2.4a} Suppose that $p>5$ is prime and $\alpha\in\Z_p$. Let $a=\langle \alpha\rangle_p$ and $t=(\alpha-a)/p$. Then
\begin{align}\label{p-1S}
S_{p-1}(\alpha)-S_{p-1}(\alpha-a)\equiv2p^2t(t+1)H_a^{(5)}-2\sum_{k=1}^{a}\frac{1}{(pt+k)^3}\pmod{p^3},
\end{align}
and
\begin{align}\label{p-12S}
&S_{\frac{p-1}2}(\alpha)-S_{\frac{p-1}2}(\alpha-a)\notag\\
\equiv&2ptH_{a}^{(4)}-8p^2t^2H_a^{(5)}+2p^2t\sum_{k=1}^{a}\frac{2H_{2k}-H_k}{k^4}-2\sum_{k=1}^{a}\frac{1}{(pt+k)^3}\pmod{p^3}.
\end{align}
 \end{lemma}
\begin{proof}
For any $n\geq 1$, by Lemma \ref{Lem2.2}, we have
\begin{align*}
&S_n(\alpha)-S_n(\alpha-1)=\sum_{k=1}^n\frac{H_k^{(2)}}k\left(\binom{\alpha}k\binom{-1-\alpha}k-\binom{\alpha-1}k\binom{-\alpha}k\right)\\
&=\frac2\alpha\sum_{k=1}^n\binom{\alpha}{k}\binom{-\alpha}kH_k^{(2)}=-\frac2{\alpha^3}+\frac2\alpha\binom{\alpha-1}{n}\binom{-\alpha-1}n\left(\frac{1}{\alpha^2}+H_n^{(2)}\right).
\end{align*}
Thus, by Lemma \ref{Lem2.3} and the fact $H_{p-1}^{(2)}\equiv0\pmod p$, we have
\begin{align*}
&S_{p-1}(\alpha)-S_{p-1}(\alpha-a)=\sum_{k=0}^{a-1}(S_{p-1}(\alpha-k)-S_{p-1}(\alpha-k-1))\notag\\
=&2\sum_{k=0}^{a-1}\left(\binom{\alpha-k-1}{p-1}\binom{-\alpha+k-1}{p-1}\left(\frac{1}{(\alpha-k)^3}+\frac{H_{p-1}^{(2)}}{\alpha-k}\right)-\frac{1}{(\alpha-k)^3}\right)\notag\\
\equiv&2\sum_{k=1}^{a}\left(\binom{pt+k-1}{p-1}\binom{-pt-k-1}{p-1}\frac{1}{(pt+k)^3}-\frac{1}{(pt+k)^3}\right)\notag\\
\equiv&2p^2t(t+1)H_a^{(5)}-2\sum_{k=1}^{a}\frac{1}{(pt+k)^3}\pmod{p^3}.
\end{align*}
Similarly, using Lemma \ref{Lem2.3} and the fact $H_{\frac{p-1}{2}}^{(2)}\equiv0\pmod p$, we get
\begin{align*}
&S_{\frac{p-1}2}(\alpha)-S_{\frac{p-1}2}(\alpha-a)=\sum_{k=0}^{a-1}(S_{p-1}(\alpha-k)-S_{p-1}(\alpha-k-1))\notag\\
=&2\sum_{k=0}^{a-1}\left(\binom{\alpha-k-1}{\frac{p-1}2}\binom{-\alpha+k-1}{\frac{p-1}2}\left(\frac{1}{(\alpha-k)^3}+\frac{H_{\frac{p-1}{2}}^{(2)}}{\alpha-k}\right)-\frac{1}{(\alpha-k)^3}\right)\notag\\
=&2\sum_{k=1}^{a}\left(\binom{pt+k-1}{\frac{p-1}2}\binom{-pt-k-1}{\frac{p-1}2}\left(\frac{1}{(pt+k)^3}+\frac{H_{\frac{p-1}{2}}^{(2)}}{pt+k}\right)-\frac{1}{(pt+k)^3}\right)\notag\\
\equiv&2\sum_{k=1}^{a}\frac1{(pt+k)^3}\frac{pt}k\left(1-\frac{pt}k+2pH_{2k}-pH_k\right)+2pt(H_{\frac{p-1}{2}}^{(2)})^2-2\sum_{k=1}^{a}\frac{1}{(pt+k)^3}\notag\\
\equiv&2ptH_{a}^{(4)}-8p^2t^2H_a^{(5)}+2p^2t\sum_{k=1}^{a}\frac{2H_{2k}-H_k}{k^4}-2\sum_{k=1}^{a}\frac{1}{(pt+k)^3}\pmod{p^3}.
\end{align*}

\end{proof}

\begin{lemma}\label{Lem2.4} Suppose that $p>7$ is prime and $\alpha\in\Z_p$. Let $a=\langle \alpha\rangle_p$ and $t=(\alpha-a)/p$. Then
\begin{align*}
S_{p-1}(\alpha-a)\equiv2p^2t^2B_{p-5}-\frac{2}5p^2tB_{p-5}\pmod{p^3},
\end{align*}
\begin{align*}
S_{\frac{p-1}2}(\alpha-a)\equiv4p^2t^2B_{p-5}-\frac{31}5p^2tB_{p-5}\pmod{p^3},
\end{align*}
\end{lemma}
\begin{proof} It is easy to see that
\begin{align*}
&\binom{pt}k\binom{-1-pt}k=\frac{pt}k\binom{pt-1}{k-1}\binom{-1-pt}k\\
&\equiv\frac{pt}k(-1)^{k-1}(1-ptH_{k-1})(-1)^k(1+ptH_k)\equiv-\frac{pt}k\left(1+\frac{pt}k\right)\pmod{p^3}.
\end{align*}
Thus, using the fact $$2H_{p-1}(2,2)=(H_{p-1}^{(2)})^2-H_{p-1}^{(4)}$$ and $H_{p-1}^{(2)}\equiv0\pmod p$, we have
\begin{align*}
&S_{p-1}(pt)=\sum_{k=1}^{p-1}\frac{H_k^{(2)}}{k}\binom{pt}k\binom{-1-pt}k\equiv-pt\sum_{k=1}^{p-1}\frac{H_k^{(2)}}{k^2}-p^2t^2\sum_{k=1}^{p-1}\frac{H_k^{(2)}}{k^3}\\
&\equiv-pt(H_{p-1}(2,2)+H_{p-1}^{(4)}-p^2t^2(H_{p-1}(2,3)+H_{p-1}^{(5)})\\
&\equiv-\frac{pt}2H_{p-1}^{(4)}-p^2t^2(H_{p-1}(2,3)+H_{p-1}^{(5)})\pmod{p^3}.
\end{align*}
In view of (\ref{Hp1r}) and (\ref{Hp1rs}), we immediately obtain the desired result
$$
S_{p-1}(pt)\equiv2p^2t^2B_{p-5}-\frac{2}5p^2tB_{p-5}\pmod{p^3}.
$$
Similarly,
$$
S_{\frac{p-1}2}(pt)\equiv-\frac{pt}2H_{\frac{p-1}2}^{(4)}-p^2t^2\left(H_{\frac{p-1}2}(2,3)+H_{\frac{p-1}2}^{(5)}\right)\pmod{p^3}.
$$
In view of (\ref{Hp12r}) and (\ref{Hp12rs}), we immediately get that
$$
S_{\frac{p-1}2}(\alpha-a)=S_{\frac{p-1}2}(pt)\equiv4p^2t^2B_{p-5}-\frac{31}5p^2tB_{p-5}\pmod{p^3}.
$$
\end{proof}
Now we are ready to complete the proof of Theorem \ref{Th1.1}.
\begin{proof}[Proof of Theorem \ref{Th1.1}]
By Lemmas \ref{Lem2.4a} and \ref{Lem2.4}, we have
\begin{align*}
&S_{p-1}(\alpha)\equiv S_{p-1}(pt)+2\sum_{k=1}^{a}\frac{-1}{(pt+k)^3}+2p^2t(t+1)H_a^{(5)}\\
\equiv&2p^2t^2B_{p-5}-\frac{2}5p^2tB_{p-5}+2\sum_{k=1}^{a}\frac{-1}{(pt+k)^3}+2p^2t(t+1)H_a^{(5)}\\
\equiv&2p^2t^2B_{p-5}-\frac{2}5p^2tB_{p-5}-2H_a^{(3)}\notag\\
&\ \ \ \ \ \ \ \ \ \ \ +6ptH_a^{(4)}+2p^2t(1-5t)H_a^{(5)}\pmod{p^3}.
\end{align*}
Thus (\ref{1.1}) is concluded. 

Furthermore, by (\ref{p-12S}) and Lemma \ref{Lem2.4}, we have
\begin{align*}
&S_{\frac{p-1}2}(\alpha)\equiv S_{\frac{p-1}2}(pt)+2\sum_{k=1}^{a}\frac{-1}{(pt+k)^3}+2ptH_a^{(4)}-8p^2t^2H_a^{(5)}\\
&\ \ \ \ \ \ \ \ \ \ \ \ \ \ \ \   +2p^2t\sum_{k=1}^{a}\frac{2H_{2k}-H_k}{k^4}\\
&\equiv4p^2t^2B_{p-5}-\frac{31}5p^2tB_{p-5}-2H_a^{(3)}+8ptH_a^{(4)}\notag\\
&\ \ \ \ \ \ \ \ \  -20p^2t^2H_a^{(5)}+2p^2t\sum_{k=1}^{a}\frac{2H_{2k}-H_k}{k^4}\pmod{p^3}.
\end{align*}
This proves (\ref{1.2}).
\end{proof}

\section{Proof of Theorem \ref{SunCT}}
\setcounter{lemma}{0}
\setcounter{theorem}{0}
\setcounter{corollary}{0}
\setcounter{remark}{0}
\setcounter{equation}{0}
\setcounter{conjecture}{0}

In this section, with the help of Theorem \ref{Th1.1}, we shall confirm Sun's conjectures (\ref{1.3a}) and (\ref{1.3}).
\begin{lemma}\label{LemH} For any prime $p>7$, we have
$$
H_{\frac{p-1}2}^{(3)}\equiv6\frac{H_{p-1}}{p^2}-\frac{81}{10}p^2B_{p-5}\pmod {p^3}.
$$
\end{lemma}
\begin{proof} In view of \cite[pp. 17]{s2000}, we have
\begin{align*}
H_{\frac{p-1}2}^{(3)}=\sum_{k=1}^{\frac{p-1}2}\frac1{k^3}\equiv-\frac{93}8p^2B_{\varphi(p^3)-4}+6\frac{B_{\varphi(p^3)-2}}{\varphi(p^3)-2}\pmod{p^3}.
\end{align*}
And by \cite[(1.2)]{s2000}, we have
\begin{align*}
\frac{B_{\varphi(p^3)-2}}{\varphi(p^3)-2}\equiv&\binom{p^2-1}2\frac{B_{3p-5}}{3p-5}-(p^2-1)(p^2-3)\frac{B_{2p-4}}{2p-4}\\
&+\binom{p^2-2}2(1-p^{p-4})\frac{B_{p-3}}{p-3}\pmod{p^3}.
\end{align*}
Then by an easy calculation, we have
$$
\frac{B_{\varphi(p^3)-2}}{\varphi(p^3)-2}\equiv\frac{B_{3p-5}}{3p-5}-3\frac{B_{2p-4}}{2p-4}+3\frac{B_{p-3}}{p-3}\pmod{p^3}.
$$
Since $p>7$, so $p-3>4$, in view of \cite[(5.2)]{s2000}, we have 
$$
B_{\varphi(p^3)-4}\equiv\frac45B_{p-5}\pmod p.
$$
Therefore,
$$
H_{\frac{p-1}2}^{(3)}\equiv6\left(\frac{B_{3p-5}}{3p-5}-3\frac{B_{2p-4}}{2p-4}+3\frac{B_{p-3}}{p-3}\right)-\frac{93}{10}p^2B_{p-5}\pmod {p^3}.
$$
In view of \cite[Theorem 2.1]{Tjnt}, we have
$$
\frac{H_{p-1}}{p^2}\equiv\frac{B_{3p-5}}{3p-5}-3\frac{B_{2p-4}}{2p-4}+3\frac{B_{p-3}}{p-3}-\frac15p^2B_{p-5}\pmod{p^3}.
$$
Thus we immediately get the desired result.
\end{proof}
\begin{lemma}
\noindent(i) Suppose that $r\in\N$ and $p\geq r+2$ is prime. Then
\begin{equation}\label{Hp1mr}
H_{p-1}(-r)\equiv\begin{cases}
-\frac{2(1-2^{p-r})}{r}B_{p-r}\pmod{p},&\text{if }r\text{ is odd},\\
-\frac{r(1-2^{p-r-1})}{r+1}pB_{p-r-1}\pmod{p},&\text{if }r\text{ is even}.
\end{cases}
\end{equation}

\noindent(ii) Suppose that $r,s\in\N$ and $p\geq r+s+2$ is prime. If $r+s$ is odd, then
\begin{equation}\label{Hp1mrs}
H_{p-1}(-r,s)\equiv H_{p-1}(r,-s)\equiv 
\frac{1-2^{p-r-s}}{r+s}B_{p-r-s}\pmod{p}.
\end{equation}
\end{lemma}
\begin{proof}
(\ref{Hp1mr}) and (\ref{Hp1mrs}) follow from \cite[Corollary 2.3]{TZ} and \cite[Theorem 3.1]{TZ} respectively.
\end{proof}
\begin{lemma}\label{Lem2.5} Let $p>7$ be a prime. Then
$$
\sum_{k=1}^{\frac{p-1}2}\frac{2H_{2k}-H_k}{k^4}\equiv\frac{31}2B_{p-5}\pmod p.
$$
\end{lemma}
\begin{proof} It is easy to see that
\begin{align*}
&\sum_{k=1}^{\frac{p-1}2}\frac{H_{2k}}{k^4}=8\sum_{k=1}^{p-1}\frac{(1+(-1)^k)H_k}{k^4}\\
=&8(H_{p-1}(1,4)+H_{p-1}^{(5)}+H_{p-1}(1,-4)+H_{p-1}(-5)).
\end{align*}
In view of (\ref{Hp1r}), (\ref{Hp1rs}), (\ref{Hp1mr}) and (\ref{Hp1mrs}), we have
$$
\sum_{k=1}^{\frac{p-1}2}\frac{H_{2k}}{k^4}\equiv\frac{13}2B_{p-5}\pmod p.
$$
Similarly,
$$
\sum_{k=1}^{\frac{p-1}2}\frac{H_{k}}{k^4}=H_{\frac{p-1}{2}}(1,4)+H_{\frac{p-1}2}^{(5)}\equiv-\frac52B_{p-5}\pmod p.
$$
Hence
$$
\sum_{k=1}^{\frac{p-1}2}\frac{2H_{2k}-H_k}{k^4}\equiv\frac{31}2B_{p-5}\pmod p.
$$
This completes the proof of Lemma \ref{Lem2.5}.
\end{proof}

Now we are ready to prove (\ref{1.3a}) and (\ref{1.3}).
For any prime $p>7$, applying (\ref{Hp1r}) and (\ref{1.1}) with $\alpha=-1/2$, $a=(p-1)/2$ and $t=-1/2$, we obtain that
\begin{align}\label{1.3z}
\sum_{k=1}^{p-1}\frac{H_k^{(2)}}{k16^k}\cdot\binom{2k}k^2=S_{p-1}\bigg(-\frac12\bigg)\equiv-2H_{\frac{p-1}2}^{(3)}-\frac{31}2p^2B_{p-5}\pmod{p^3}.
\end{align}
In view of (\ref{1.3z}) and Lemma \ref{LemH},  we immediately obtain that
$$
\sum_{k=1}^{p-1}\frac{H_k^{(2)}}{k16^k}\cdot\binom{2k}k^2\equiv-12\frac{H_{p-1}}{p^2}+\frac7{10}p^2B_{p-5}\pmod{p^3},
$$ 
i.e., (\ref{1.3a}) is valid when $p>7$.
Of course, (\ref{1.3a}) can be verified easily  for $p=5,7$.

Let us turn to (\ref{1.3}). We can check (\ref{1.3}) directly when $p=5,7$. Suppose that $p>7$. Combining (\ref{1.1}) and (\ref{1.2}), we have
\begin{align}\label{Sp1p12alpha}
&S_{p-1}(\alpha)-S_{\frac{p-1}2}(\alpha)\equiv -2p^2t^2B_{p-5}+\frac{29}5p^2tB_{p-5}+2p^2t(t+1)H_a^{(5)}\notag\\
& -2ptH_a^{(4)}+8p^2t^2H_a^{(5)}-2p^2t\sum_{k=1}^{a}\frac{2H_{2k}-H_k}{k^4}\pmod{p^3}.
\end{align}
Applying (\ref{Hp12r}) and (\ref{Sp1p12alpha}) with $\alpha=-1/2$, $a=(p-1)/2$ and $t=-1/2$, we have
$$
\sum_{k=\frac{p+1}2}^{p-1}\frac{\binom{2k}k^2}{k16^k}H_k^{(2)}=S_{p-1}\left(-\frac12\right)-S_{\frac{p-1}2}\left(-\frac12\right)\equiv\frac{31}2p^2B_{p-5}\pmod{p^3},
$$
The proof of Theorem \ref{SunCT} is complete.\qed

\section{Proof of Theorem \ref{Th1.3}}
\setcounter{lemma}{0}
\setcounter{theorem}{0}
\setcounter{corollary}{0}
\setcounter{remark}{0}
\setcounter{equation}{0}
\setcounter{conjecture}{0}
\noindent {\it Proof of Theorem \ref{Th1.3}}. It is easy to see that
$$
\binom{-\alpha}k=\frac{\alpha}{\alpha+k}\binom{-\alpha-1}k.
$$
So by Lemma \ref{Lem2.2}, we have
\begin{align*}
&\sum_{k=1}^{p-1}\binom{\alpha}k\binom{-\alpha-1}k\frac{H_k^{(2)}}{\alpha+k}=\frac1\alpha\sum_{k=1}^{p-1}\binom{\alpha}k\binom{-\alpha}kH_k^{(2)}\\
&=-\frac1{\alpha^3}+\frac1{\alpha}\binom{\alpha-1}{p-1}\binom{-\alpha-1}{p-1}\left(\frac1{\alpha^2}+H_{p-1}^{(2)}\right).
\end{align*}
We know that $a=\langle a\rangle_p+pt$, so set $k=\langle a\rangle_p$ in Lemma \ref{Lem2.3} and by (i), we have
\begin{align*}
&\sum_{k=1}^{p-1}\binom{\alpha}k\binom{-\alpha-1}k\frac{H_k^{(2)}}{\alpha+k}\\
&\equiv-\frac1{\alpha^3}+\frac{p^2t(t+1)}{\alpha\langle a\rangle_p^2}\bigg(\frac1{\alpha^2}+\frac{2pH_{\langle a\rangle_p}}{\alpha^2}-\frac{p(2t+1)}{\alpha^2\langle a\rangle_p}+\frac23pB_{p-3}\bigg)\pmod{p^4}.
\end{align*}
Similarly, by Lemma \ref{Lem2.2}, Lemma \ref{Lem2.3} and (ii), we have
\begin{align*}
&\sum_{k=1}^{\frac{p-1}2}\binom{\alpha}k\binom{-\alpha-1}k\frac{H_k^{(2)}}{\alpha+k}=\frac1a\sum_{k=1}^{\frac{p-1}2}\binom{\alpha}k\binom{-\alpha}kH_k^{(2)}\\
&=-\frac1{\alpha^3}+\frac1{\alpha}\binom{\alpha-1}{\frac{p-1}2}\binom{-\alpha-1}{\frac{p-1}2}\left(\frac1{\alpha^2}+H_{\frac{p-1}2}^{(2)}\right)\\
&\equiv-\frac1{\alpha^3}+\frac{pt}{\alpha\langle a\rangle_p}\left(\frac1{\alpha^2}+\frac73pB_{p-3}-\frac{pt}{\alpha^2\langle a\rangle_p}+\frac{2p}{\alpha^2}\sum_{k=1}^{\langle a\rangle_p}\frac1{2k-1}\right)\pmod{p^3}.
\end{align*}
\noindent Now the proof of Theorem \ref{Th1.3} is finished.\qed

\end{document}